\documentclass{amsproc}%
\usepackage{amsfonts}%
\usepackage{amsmath}%
\usepackage{amssymb}%
\usepackage{graphicx}
\providecommand{\U}[1]{\protect \rule{.1in}{.1in}}
\theoremstyle{plain}

\newtheorem{lemma}{Lemma}

\newtheorem{theorem}{Theorem}
\numberwithin{equation}{section}
\begin{document}
\title[Local Derivations]{Local derivations on Rings containing a von Neumann algebra and a question of Kadison.}
\author{Don Hadwin}
\address{Department of Mathematics, University of New Hampshire, Durham, NH 03824, USA }
\email{don@unh.edu}
\urladdr{http://euclid.unh.edu/\symbol{126}don}
\author{Jiankui Li}
\address{Department of Mathematics, East China University of Science and Technology,
Shanghai 200237, China}
\email{jiankuili@yahoo.com}
\author{Qihui Li}
\address{Department of Mathematics, East China University of Science and Technology,
Shanghai 200237, China}
\email{lqh991978@gmail.com}
\author{Xiujuan Ma}
\address{Department of Mathematics, Hebei University of Technology, Tianjing, 300130, China}
\email{mxjsusan@hebut.edu.cn}

\begin{abstract}
We prove that if $\mathcal{M}$ is a von Neumann algebra whose abelian summand
is discrete, then every local derivation on the algebra of all measurable
operators affilated with $\mathcal{M}$ is a derivation. This answers a
question of Richard Kadison.

\end{abstract}
\maketitle

\bigskip

At a conference held in Texas A $\&$ M University of 2012, Richard Kadison
gave a talk about his joint work with Zhe Liu [10, 11] in which they proved
that the only derivation that maps the algebra $S(\mathcal{M})$ of closed
densely defined operators affiliated  with a factor von Neumann algebra
$\mathcal{M}$ of type II$_{1}$ into that von Neumann algebra is zero. Kadison
asked whether every local derivation on $S(\mathcal{M})$ is a derivation.

In this note we prove a general ring-theoretic result which implies that for
all von Neumann algebras $\mathcal{M}$ whose abelian summand is discrete,
every local derivation on $S(\mathcal{M})$ of all measurable operators is a derivation.

If $\mathcal{R}$ is a ring (resp. algebra) and $\delta:\mathcal{R}%
\rightarrow \mathcal{R}$ is an additive (resp. linear) mapping, we say that
$\delta$ is a \emph{derivation} if, for all $a,b\in \mathcal{R}$, we have
\[
\delta \left(  ab\right)  =\delta \left(  a\right)  b+a\delta \left(  b\right)
.
\]
We say that an additive mapping $\delta$ is a \emph{local derivation} if, for
every $x\in \mathcal{R}$ there is a derivation $\rho_{x}$ on $\mathcal{R}$ such
that
\[
\delta \left(  x\right)  =\rho_{x}\left(  x\right)  .
\]

To prove our main result, we need two lemmas. The first is a result of [3,
Corollary 4.5]. Suppose $\mathcal{R}$ is a ring with identity $1$ and $n\geq2$
is an integer. We say that a subset $\left \{  E_{ij}:1\leq i,j\leq n\right \}
\subset \mathcal{R}$ is a \emph{system of }$n\times n$ \emph{matrix units for
}$\mathcal{R}$ if and only if $\sum_{i=1}^{n}E_{ii}=1$ and, for $1\leq
i,j,s,t\leq n$, $E_{ij}E_{st}=0$ if $j\neq s$ and $E_{ij}E_{st}=E_{it}$ if
$j=s$.

In [3, p.11], Bresar shows that for any unital ring $\mathcal{B}$, the ring
$M_{n}(\mathcal{B})$ is generated by the set of all idempotents in
$\mathcal{B}$, where $2\le n$.

The following Lemma 1 is a special case of [3, Corollary 4.5].

\begin{lemma}
If $\mathcal{R}$ is a ring with identity $1$ that possesses a system of
$n\times n$ matrix units for some $n\geq2,$ then every local derivation on
$\mathcal{R}$ is a derivation.
\end{lemma}

Note that, in general, derivations need not leave ideals invariant, e.g.,
differentiation on the polynomials. However, if $p$ a central idempotent, then
every derivation (hence, every local derivation) $\delta$ leaves
$p\mathcal{R}$ and $\left(  1-p\right)  \mathcal{R}$ invariant, since%
\[
\delta \left(  pa\right)  =\delta \left(  \left(  pa\right)  p\right)
=pa\delta \left(  p\right)  +\delta \left(  pa\right)  p.
\]
Moreover, $p\mathcal{R}$ is isomorphic to $\mathcal{R}/\left(  1-p\right)
\mathcal{R}$. This yields the following corollary to the preceding lemma. A
family $\mathcal{P}$ of central idempotents for a ring $\mathcal{R}$ is
\emph{separating} if and only if, for each nonzero $x\in \mathcal{R}$, there is
a $p\in \mathcal{P}$ such that $px\neq0.$

\begin{lemma}
Suppose $\mathcal{R}$ is a ring with identity and $\mathcal{P}$ is a
separating family of central idempotents such that, for each $p\in \mathcal{P}%
$, every local derivation on $p\mathcal{R}$ is a derivation. Then every local
derivation on $\mathcal{R}$ is a derivation.
\end{lemma}

\begin{proof}
Suppose $\delta:\mathcal{R}\rightarrow \mathcal{R}$ is a local derivation,
$a\in \mathcal{R}$ and $p\in \mathcal{P}$. Then there is a derivation $\rho$ on
$\mathcal{R}$ such that%
\[
\delta \left(  pa\right)  =\rho \left(  pa\right)  =\rho \left(  p\left(
pa\right)  \right)  =p\rho \left(  pa\right)  +\rho \left(  p\right)
pa=p\left[  \rho \left(  pa\right)  +\rho \left(  p\right)  a\right]
=pgr\left(  pa\right)  =p\delta \left(  pa\right)  .
\]
It follows that $\delta \left(  p\mathcal{R}\right)  \subset p\mathcal{R}$ and
that $\delta|p\mathcal{R}$ is a local derivation. Hence $\delta|p\mathcal{R}$
is a derivation. Thus, for every $a,b\in \mathcal{R}$ and every $p\in
\mathcal{P}$, we have%
\[
p\delta \left(  ab\right)  =\delta \left(  pab\right)  =\delta \left(  \left(
pa\right)  \left(  pb\right)  \right)  =
\]%
\[
pa\delta \left(  pb\right)  +\delta \left(  pa\right)  pb=p\left[
a\delta \left(  b\right)  +\delta \left(  a\right)  b\right]  .
\]
Hence%
\[
p\left[  \delta \left(  ab\right)  -\left[  a\delta \left(  b\right)
+\delta \left(  a\right)  b\right]  \right]  =0.
\]
Since $\mathcal{P}$ is separating, we see that $\delta$ is a derivation on
$\mathcal{R}$.
\end{proof}

\bigskip

An abelian von Neumann algebra $\mathcal{M}$ is discrete if it is generated by
its minimal nonzero projections; equivalently, if the identity operator in
$\mathcal{M}$ is the sum of the minimal projections in $\mathcal{M}$. Since
$\mathcal{M}$ is abelian, it follows that $Q\mathcal{M}=\mathbb{C}Q$ for every
minimal projection $Q$ in $\mathcal{M}$. Every von Neumann algebra on a
Hilbert space $H$ is the direct sum of algebras $\mathcal{M}_{n}$ with $1\leq
n<\infty$ (the finite type $I_{n}$ summands) and a von Neumann algebra
$\mathcal{M}_{\infty}$ which is the direct sum of algebras of type $I_{\infty
},$ $II,$ and $III$ (not all summands need be present). Call the corresponding
central projections $P_{n}$ with $1\leq n\leq \infty$.

\bigskip

\begin{theorem}
Suppose $\mathcal{M}$ is a von Neumann algebra on a Hilbert space $H$. Then

\begin{enumerate}
\item If $\mathcal{M}_{1}=0$ and $\mathcal{R}$ is a ring containing
$\mathcal{M}$ with the same identity as $\mathcal{M}$ such that $\mathcal{P}%
=\left \{  P_{n}:2\leq n\leq \infty \right \}  $ is a separating family of central
idempotents for $\mathcal{R}$, then every local derivation on $\mathcal{R}$ is
a derivation.

\item If $\mathcal{M}_{1}$ is discrete, then every local derivation on the
algebra of closed densely defined operators affiliated with $\mathcal{M}$ is a derivation.
\end{enumerate}
\end{theorem}

\begin{proof}
$\left(  1\right)  .$

By [9, Theorem 6.6.5], we know that $P_{n}\mathcal{M}$ contains an $n\times n$
system of matrix units for $2\leq n<\infty$, and it follows from [9, Lemma
6.5.6] that $P_{\infty}\mathcal{M}$ contains a $2\times2$ system of matrix
units. Since $P_{n}\mathcal{M}\subset P_{n}\mathcal{R}$ is a unital embedding
and $\mathcal{P}$ is separating, it follows from the two lemmas above that
every local derivation on $\mathcal{R}$ is a derivation.

$\left(  2\right)  .$ Let $\mathcal{R}$ be the algebra of all measurable
operators affiliated with $\mathcal{M}$. However, since $\mathcal{M}_{1}$ is
discrete, $P_{1}$ is the orthogonal sum of a family $\left \{  Q_{\lambda
}:\lambda \in \Lambda \right \}  $ of minimal projections and that, for each
$\lambda \in \Lambda$, $Q_{\lambda}\mathcal{M}=\mathbb{C}Q_{\lambda},$ which
means that $Q_{\lambda}\mathcal{M}^{\prime}=B\left(  Q_{\lambda}H\right)
Q_{\lambda},$ which, in turn, implies that $Q_{\lambda}\mathcal{R}%
=\mathbb{C}Q_{\lambda}$, so every local derivation on $Q_{\lambda}\mathcal{R}$
is a derivation. Since Since the elements of $\mathcal{R}$ are densely defined
operators on $H$, and $\sum_{\lambda \in \Lambda}Q_{\lambda}+\sum_{2\leq
n\leq \infty}P_{n}=1$, it follows that $\left \{  Q_{\lambda}:\lambda \in
\Lambda \right \}  \cup \left \{  P_{n}:2\leq n\leq \infty \right \}  $ is a
separating family of central idempotents for $\mathcal{R}$. Arguing as in the
proof of $\left(  1\right)  $ we can apply the lemmas to see that every local
derivation on $\mathcal{R}$ is a derivation.
\end{proof}

In [1, Theorem 3.8], the authors give necessary and sufficient conditions on a
commutative von Neumann algebra $\mathcal{M}$ for the existence of local
derivations which are not derivations on the algebra $S(\mathcal{R})$ of
measurable operators affiliated with $\mathcal{M}$.


For a von Neumann algebra $\mathcal{M}$, we can define the set $S(\mathcal{M}%
)$  of all measurable operators affiliated with $\mathcal{M}$ and the set
$LS(\mathcal{M})$ of all local measurable operators affiliated with
$\mathcal{M}$

In [12], Muratov and Chilin show  that $LS(\mathcal{M})$ is a unital $\ast
$-algebra when equipped with algebraic  operations of the strong addition,
multiplication, and taking the  adjoint of an operator and $S(\mathcal{S})$ is
a unital $\ast$-subalgebra  of $LS(\mathcal{M})$.

Suppose that $\mathcal{M}$ is a von Neumann algebra with a faithful normal
semi-finite trace $\tau$. Let $S(\mathcal{M}, \tau)$ denote the algebra of all
$\tau$-measurable operators affiliated with $\mathcal{M}$. It is clear that
$\mathcal{M }\subseteq S(\mathcal{M}) \subseteq LS(\mathcal{M})$ ( see for
details [1, 2, 13]).

In Theorem 1, if we choose that $\mathcal{R}$ is $\mathcal{M}$, $S(\mathcal{M}%
)$ or $LS(\mathcal{M})$, we can obtain [1, Theorem 2.5 and Proposition 2.7].

By Theorem 1 and [1, Theorem 3.8], we can completely answer [15, Conjecture 48].

\section*{Acknowledgement}

The first author is partially supported by a travel grant from the Simons
Foundation and the second author is supported by NSF of China.

\end{document}